\documentclass[12pt,a4paper,twoside,reqno]{amsart}

\allowdisplaybreaks

\usepackage[all,cmtip]{xy}
\usepackage{amsmath,amscd}
\usepackage{amssymb,amsfonts}
\usepackage[driver=pdftex,margin=3cm,heightrounded=true,centering]{geometry}
\usepackage{mathtools}
\usepackage{tensor}
\usepackage{url}
\usepackage[colorlinks=true,linkcolor=blue,citecolor=blue]{hyperref}
\usepackage{slashed}
\usepackage[new]{old-arrows}

\usepackage{graphicx}

\usepackage{a4wide}

\usepackage{thmtools}
\declaretheorem[style=theorem,name={Theorem}]{theoremletter}

\usepackage{amsthm}

\theoremstyle{plain}

\newtheorem{theorem}{Theorem}[section]
\newtheorem*{thm*}{Theorem}

\theoremstyle{definition}

\theoremstyle{remark} 

\newtheorem{remark}[theorem]{Remark}

\theoremstyle{plain}

\usepackage{amscd}

\numberwithin{equation}{section}

\newcommand{\alpheqn}[1][\relax]{
     \refstepcounter{equation}
     \if#1\relax \relax
       \else \label{#1}
     \fi  
     \setcounter{saveeqn}{\value{equation}}%
    \setcounter{equation}{0}%
    \renewcommand{\theequation}{\thealphequation}}
\newcommand{\reseteqn}{\setcounter{equation}{\value{saveeqn}}%
     \renewcommand{\theequation}{\thearabicequation}}

\providecommand{\mathscr}{\mathcal} 
\IfFileExists{mathrsfs.sty}{
\usepackage{mathrsfs}}         
   {
     \IfFileExists{eucal.sty}{
        \usepackage[mathscr]{eucal}} 
   {
   }
}

\newcommand{\ev}{{\operatorname{ev}}}
\newcommand{\varps}{{\varepsilon}}
\newcommand{\vertiii}[1]{{\left\vert\kern-0.25ex\left\vert\kern-0.25ex\left\vert #1 
    \right\vert\kern-0.25ex\right\vert\kern-0.25ex\right\vert}}
\newcommand{\Bvert}[1]{{\Big\vert\kern-0.25ex\Big\vert\kern-0.25ex\Big\vert #1 
    \Big\vert\kern-0.25ex\Big\vert\kern-0.25ex\Big\vert}}
\newcommand{\bvert}[1]{{\big\vert\kern-0.25ex\big\vert\kern-0.25ex\big\vert #1 
    \big\vert\kern-0.25ex\big\vert\kern-0.25ex\big\vert}}
\newcommand{\nvert}[1]{{\vert\kern-0.25ex\vert\kern-0.25ex\vert #1 
    \vert\kern-0.25ex\vert\kern-0.25ex\vert}}

\renewcommand{\leq}{\leqslant}
\renewcommand{\geq}{\geqslant}

\newcommand{\C}[1]{\mathcal{#1}}

\newcommand{\E}[1]{\emph{#1}}

\newcommand{\inn}[1]{\langle #1 \rangle}

\newcommand{\A}{{\mathcal{A}}}
\newcommand{\BB}{{\mathbb{B}}}
\newcommand{\CC}{{\mathbb{C}}}
\newcommand{\NN}{{\mathbb{N}}}
\renewcommand{\S}{{\mathcal{S}}}
\renewcommand{\E}{{\mathcal{E}}}

\makeatletter
\@namedef{subjclassname@2020}{%
  \textup{2020} Mathematics Subject Classification}
\makeatother

\begin{document}

\author{David Kyed}
\address{David Kyed, Department of Mathematics and Computer Science, University of Southern Denmark, Campusvej 55, DK-5230 Odense M, Denmark}
\email{dkyed@imada.sdu.dk}

\author{Ryszard Nest}
\address{Ryszard Nest,  Department of Mathematical Sciences, Universitetsparken 5, 2100 København {\O} \\ Denmark}
\email{rnest@math.ku.dk}



\dedicatory{Dedicated to our friend, colleague and teacher, Erik Christensen}

\title{Finiteness of metrics on state spaces}

\begin{abstract}
We show that Connes' metric on the state space associated with a spectral triple is nowhere infinite exactly when it is globally bounded. Moreover, we produce a family of simple examples showing that this is not automatically the case.
\end{abstract}

\maketitle

\section{Introduction}

  Over the past 30 years, Connes' non-commutative geometry  has  developed into an important independent branch of mathematics with strong ties to operator algebras, classical geometry, number theory and theoretical physics \cite{MarcCons:NCG-and-number-theory, Con:NCG,  Connes-Marcolli}. 
  The key object in non-commutative geometry is a so-called \emph{spectral triple}, $(\A, H, D)$, which, in  its simplest form,  consists of a Hilbert space $H$ together with a unital $*$-subalgebra $\A\subset \BB(H)$ of the bounded operators on $H$ and an unbounded, densely defined, self-adjoint operator $D$ on $H$ satisfying the following two requirements:
 \begin{itemize}
\item[(i)] Each $a\in \A$ preserves the domain of $D$ and the commutator $[D,a]=Da-aD$ extends to a bounded operator on $H$.
\item[(ii)] The operator $D$ has compact resolvent; i.e.~$(i+D)^{-1}$ is a compact operator.
\end{itemize}
The motivating example from classical geometry arises from a connected, closed,  Riemannian spin  manifold $M$, where $H$ is the Hilbert space of $L^2$-sections of the spinor bundle, $D$ is the associated Dirac operator and $\A$ is the algebra $C^\infty(M)$ of smooth functions  acting as multiplication operators on $H$. In this setting, Connes realised  \cite[Proposition 1]{Con:CFH} that the Riemannian metric $d$ on $M$ may be recovered from the associated spectral triple  by means of the formula
\begin{align}\label{eq:reconstruct-riemmanian}
d(p,q)=\sup\big\{ |f(p)-f(q)| \mid f\in C^\infty(M),  \|[D,f]\| \leq 1  \big\},
\end{align}
where  $\|[D,f]\|$ denotes the operator norm of the bounded commutator.  Denoting by $\ev_p\colon C(M)\to \CC$ the pure state on $C(M)$ given by evaluation at $p\in M$, the equation \eqref{eq:reconstruct-riemmanian} may be rewritten as

\begin{align*}
d(p,q)=\sup\big\{ |\ev_p(f)-\ev_q(f)| \mid f\in C^\infty(M),  \|[D,f]\| \leq 1  \big\},
\end{align*}
and the right hand side may therefore be interpreted as a metric on the pure states which agrees with the Riemannian metric on $M$.
For a general spectral triple $(\A, H, D)$, it is therefore natural to define a metric on the state space $\S (A)$ of the $C^*$-closure $A$ of $\A$ by setting
\begin{align}\label{eq:connes-metric}
d_D(\mu, \nu):=\sup\big\{|\mu(a)-\nu(a)| \mid a\in \A, \|[D,a]\|\leq 1 \big\}
\end{align}
for $\mu, \nu\in \S (A)$.
This construction also dates back to  \cite{Con:CFH} and is often referred to as the \emph{Connes metric} or the \emph{Monge-Kantorovi\v{c} metric} \cite{KaRu:FSE, KaRu:OSC} in the literature.
A bit of care is needed at this point, since the quantity defined in \eqref{eq:connes-metric} is only an \emph{extended metric}, in that it can attain the value infinity, but otherwise satisfies the usual axioms for a metric. When the spectral triple arises from a  spin manifold as described above, $d_D$ metrises the weak$^*$-topology on the state space $\S (A)$ (see \cite{Con:CFH} or \cite[Proposition 8.1]{KK:DCQ} for a  more detailed explanation), and since $\S (A)$ is weak$^*$ connected, it follows that $d_D$   cannot attain the value infinity and, in turn, that it assigns $\S (A)$ a finite diameter.
For a general spectral triple, the question of when $d_D$ metrises the weak$^*$-topology turns out to be rather subtle, and led Rieffel to introduce his theory of quantum metric spaces \cite{Rie:GHD, Rie:MSS, Rie:CQM, Lat:QGH, Lat:GPS}.  
The more modest question of when the extended metric $d_D$ is finite (in the sense of not attaining the value infinity) has also been investigated, but the literature available is unfortunately not optimal, and the aim of the present note is to clarify the situation.  Before proceeding, we briefly summarise the history of this problem:  In his seminal paper, Connes remarked  \cite[Proposition 4]{Con:CFH} that under the additional assumption that the set
\[
\C L_1:=\{a\in \A \mid  \|[D,a]\|\leq 1\}
\]
projects to a  bounded set in the quotient $A/\mathbb{C}$, then the metric $d_D$ is finite. However,  since the difference of two states descends to a functional of norm at most $2$ on $A/\mathbb{C}$, this condition even ensures that   $d_D$ is globally bounded; i.e.~that state space has finite diameter with respect to $d_D$. So, at least a priori, Connes' condition is stronger than the metric $d_D$ being finite but, as we will show below, this is not actually the case.  An easy necessary condition for $d_D$ to be finite is that $D$ only commutes with scalar multiples of the identity  (see the paragraph preceding the proof of Theorem \ref{thm:main-thm-B})
and in the unpublished preprint \cite[Proposition B.1]{Rennie-Varilly} it is claimed that this condition is also sufficient. This result was  later quoted in the unpublished preprint \cite{BMR:DSS}, 
but beyond these references, there seem to be no general statements along these  lines available in the literature, although earlier partial results were obtained in \cite{AL:Connes-metric}. 
 Unfortunately, the proof of  \cite[Proposition B.1]{Rennie-Varilly} contains a gap, and in Theorem \ref{thm:main-thm-B} below we provide concrete counter examples to the statement in  \cite[Proposition B.1]{Rennie-Varilly}.
 Before proceeding to Theorem \ref{thm:main-thm-B}, we first prove:

\begin{theoremletter}\label{thm:main-thm-A}
For a unital  spectral triple $(\A, H, D)$, the metric $d_D$ is finite if and only if it assigns a finite diameter to the state space. 
\end{theoremletter}
As already indicated, Theorem \ref{thm:main-thm-A} implies that the criterion from \cite[Proposition 4]{Con:CFH} discussed above is both necessary and sufficient; see Remark \ref{rem:connes-criterion} for more details. The proof of Theorem \ref{thm:main-thm-A}  is a simple application of the Principle of Uniform Boundedness, but despite its simple proof,  we have not been able to find the statement  anywhere in the literature and have therefore chosen to single it out here as a separate result for future reference. 
 Although formulated in the language of spectral triples, the statement in  Theorem \ref{thm:main-thm-A}  holds true, more generally, in the setting  of Lipschitz semi-norms on operator systems, and we elaborate on this aspect in Remark \ref{rem:Lip-rem} below.\\
 
 As alluded to above, the second result of the present note is the following:
 
 \begin{theoremletter}\label{thm:main-thm-B}
 For any strictly increasing  sequence $(d_n)_{n\in \NN}\subset  [0,\infty)$ with  $\lim_{n\to \infty} d_n=\infty$ there exists a unital spectral triple $(\A, H, D)$ such that: 
\begin{itemize}
\item[(a)] The spectrum of $D$ agrees with the closure of the set $\{d_n \mid n\in \NN \}$.
\item[(b)] The algebra $\A$ is commutative.
\item[(c)] The only elements from $\A$ that commute with $D$ are the scalar multiples of the unit.
\item[(d)] The extended metric $d_D$ on the state space attains the value infinity. 
\end{itemize}
\end{theoremletter}
The  spectral triples constructed in Theorem \ref{thm:main-thm-B} thus  provide concrete counter examples to the statement in \cite[Proposition B.1]{Rennie-Varilly}. 
Note also, that combining Theorem \ref{thm:main-thm-A} with \cite[Theorem 5.4 \& 5.5]{PutJul:subshifts} provides additional, and perhaps more natural, counter examples, for which  it is, however,  less clear how to directly access states whose distance is infinite.
Lastly, note that  a similar result for  non-unital  spectral triples is obtained in \cite[Proposition 3.10]{CAM-Moyal}. For these, 
one may of course pass to the minimal unitization \cite[Lemma 5.4]{MT13}, but this comes at the cost of loosing the compactness of the resolvent.

\subsection*{Acknowledgements}
The authors gratefully acknowledge the financial support from  the Independent Research Fund Denmark through grant no.~9040-00107B and 1026-00371B, and from the the ERC through the MSCA Staff Exchanges grant no.~101086394.
Moreover, they are grateful to Erik Christensen and Jens Kaad for many illuminating discussions on  matters related to metrics on state spaces, to Max Holst Mikkelsen for pointing out an inaccuracy in a preliminary version of the paper, to Pierre Martinetti and Jean-Christophe Wallet for making them aware of the references  \cite{CAM-Moyal} and \cite{MT13} and to the anonymous referee for helpful clarifications.

\section{Proofs}
We first carry out the short proof of  Theorem \ref{thm:main-thm-A}.

\begin{proof}[Proof of Theorem \ref{thm:main-thm-A}]
The backwards implication is obvious, so we only need to prove that if $d_D$ does not attain the value infinity then it assigns a finite diameter to the state space $\S(A)$, where $A$, as above, denotes the $C^*$-closure of $\A$ in $\BB(H)$. Fix a state $\mu_0\in \S(A)$. By the triangle inequality, it suffices to show that there exists a constant $C>0$ such that $d_D(\mu, \mu_0)\leq C$ for any other $\mu \in \S(A)$. Note that
\begin{align*}
d_D(\mu, \mu_0)&=\sup\{ |\mu(a)-\mu_0(a)| \mid a\in \A,  \|[D,a]\|\leq 1 \} \\
& =\sup\{ |\mu(a-\mu_0(a)1)| \mid a\in \A,  \|[D,a]\|\leq 1 \}.
\end{align*}
Since $[D, a-\mu_0(a)1]=[D,a]$, we have

\[
\left\{a-\mu_0(a)1 \mid \|[D,a]\|\leq 1\right\}=\left\{a\in \A \mid \mu_0(a)=0 \ \text{ and } \ \|[D,a]\|\leq 1 \right\}=:\E
\]
and our assumption therefore implies that $\{\mu(a) \mid a \in \E\}$ is bounded for any $\mu\in \S(A)$. Since any bounded functional $\varphi\in A^*$ is a linear combination of four elements in $\S(A)$ \cite[Corollary 4.3.7]{KaRi:FTO} it follows that $\{\varphi(a) \mid a \in \E\}$ is bounded for all $\varphi\in A^*$, and by the Principle of Uniform Boundedness (see e.g.~\cite[Proposition 8.11]{MeiseVogt}) we obtain the $\E$ is bounded in the operator norm $\|\cdot \|$. As states have norm 1, it therefore follows that
\[
d_D(\mu,\mu_0)=\sup\{ |\mu(a)| \mid a\in \E\}\leq \sup\{\|a\| \mid a\in \E \}<\infty,
\]
and $C:=\sup\{\|a\| \mid a\in \E \}$ therefore satisfies the desired inequality.
\end{proof}

\begin{remark}\label{rem:Lip-rem}
We have chosen to formulate Theorem  \ref{thm:main-thm-A} within the realm of spectral triples, but it should be remarked that the result holds true, more generally, in the setting of operator systems with Lipschitz seminorms (cf.~\cite[Definition 2.1]{KK:quantumSU2}), which constitutes the basis in (one version of) Rieffel's theory of compact quantum metric spaces \cite{Rie:MSS, Rie:CQM}. 
More precisely, if $X$ is a complete operator system and $L\colon X\to [0,\infty]$ is a Lipschitz seminorm, then by repeating the proof of Theorem \ref{thm:main-thm-A} verbatim, we obtain that  the extended metric defined on the state space $\S(X)$ by the formula
\[
d_L(\mu, \nu):=\sup\{|\mu(x)-\nu(x)| \mid x\in X, L(x)\leq 1\}
\]
is finite  if and only if it assigns a finite diameter to $\S(X)$.  To make the link with the above constructions explicit, one obtains a Lipschitz seminorm from a spectral triple $(\A,H, D)$ by setting $L(a)=\|[D,a]\|$ for $a\in \A$ and extending $L$ to the $C^*$-closure $A$ by the value infinity.

\end{remark}

\begin{remark}\label{rem:connes-criterion}
If $A$ is a unital $C^*$-algebra and $\mu_0\in \S(A)$ is a state, then one obtains an isomorphism of Banach spaces $\Phi\colon A\simeq \CC \oplus A/\CC$ given by $a\mapsto (\mu_0(a), \pi(a))$ where $\pi\colon A\to A/\CC$ denotes the quotient map.
Returning to the setting of the proof of Theorem \ref{thm:main-thm-A}, we saw that finiteness of the metric $d_D$ implies norm-boundedness of the set
\[
\E:=\{a\in \A \mid \mu_0(a)=0 \ \text{ and } \ \|[D,a]\|\leq 1 \}
\]
which, in turn, implies norm-boundedness of $\Phi(\E)=\{0\} \times \{\pi(a) \mid \|[D,a]\|\leq 1\}$.  Hence Connes' sufficient condition for finiteness of $d_D$ from \cite[Proposition 4]{Con:CFH} is also necessary. 
\end{remark}
Before embarking on the proof of Theorem  \ref{thm:main-thm-B}, we briefly explain why its condition (c) is relevant.  
Assume therefore that $(\A,H,  D)$ is a spectral triple and that there exists $a\in \A\setminus \CC 1$ such that $[D,a]=0$. If $\mu(a)=\nu(a)=:\lambda$ for all $\mu,\nu\in \S(A)$, then all states vanish on $\lambda 1-a$ which cannot happen since $a\neq \lambda 1$ \cite[Theorem 4.3.4]{KaRi:FTO}. Hence, there exist $\mu,\nu\in \S(A)$ with $\mu(a)\neq \nu(a)$ and since $D$ also commutes with all scalar multiples of $a$, we conclude that
\[
d_D(\mu, \nu) \geq |\mu(Ra)-\nu(Ra)|=R|\mu(a)-\nu(a)|
\]
for all $R>0$ and hence that $d_D(\mu,\nu)=\infty$.  This argument is well known and can also be found in the existing literature, but  we opted to include it here for the readers convenience.  Condition (c) of Theorem \ref{thm:main-thm-B} is therefore an obvious necessary condition for finiteness of the Connes metric, but as the theorem shows, this is not sufficient. We now proceed with the proof of Theorem \ref{thm:main-thm-B}. 

\begin{proof}[Proof of Theorem \ref{thm:main-thm-B}]
Denote by $H$ the Hilbert space $\ell^2(\NN)$ and by $(e_n)_{n\in \NN}$ its natural orthonormal basis. On $H_0:=\text{span}_{\mathbb{C}}\{e_n\mid n\in \NN\}$ we define $D_0\colon H_0 \to H$ by $D_0e_n:= d_n e_n$. It is easy to see that $D_0$ is unbounded and essentially self-adjoint (cf.~\cite[Proposition 3.8]{Schm:Unbd}), and we denote by $D$ its self-adjoint closure. 
The operator $(i+D)^{-1}$ extends $(i+D_0)^{-1}$,  given by  $(i+D_0)^{-1}e_n=(i+d_n)^{-1}e_n$, and since $(i+d_n)^{-1}\underset{n\to \infty}{\longrightarrow} 0$ it follows that $(i+D)^{-1}$ is indeed a compact operator in $H$. Moreover, the spectrum of $D$ (cf.~\cite[Section 2.2]{Schm:Unbd}) is easily seen to be equal to the closure of $\{d_n\mid n\in \NN\}$, and we are therefore left with the task of constructing a unital $*$-algebra $\A\subset \BB(H)$ such that $(\A, H, D)$ is a spectral triple satisfying (b), (c) and (d).\\
To this end, choose a sequence $\varps_n>0$ decaying fast enough that 
 \begin{align}\label{eq:varps-sequence}
 \lim_{n\to \infty} (\varps_n2^{n/2}(d_{n+1}-d_n))=0,
 \end{align}
  and consider the unit vectors in $H$ defined as
\[
f_n:=\frac{1}{\sqrt{1+ \varps_n^2}}(e_n + \varps_ne_{n+1}).
\]
We denote by $q_n$ the orthogonal projection onto the subspace spanned by $f_n$. Since $f_n\in H_0\subseteq \text{Dom}(D)$, we automatically obtain that  $q_n$ preserves the domain of $D$ and we now show that $[D,q_n]$ extends boundedly to $H$. For $\xi \in H_0$, we compute as follows:
\begin{align*}
Dq_n (\xi) &= D(\inn{\xi, f_n }f_n)\\
&= \frac{1}{1+ \varps_n^2}  D\Big(\inn{\xi, e_n}e_n +\varps_n\inn{\xi, e_{n+1}}e_n +\varps_n\inn{\xi, e_n}e_{n+1} +\varps_n^{2} \inn{\xi, e_{n+1}}e_{n+1}\Big)\\
&= \frac{1}{1+ \varps_n^2}\Big(\inn{\xi, e_n}d_ne_n +\varps_n\inn{\xi, e_{n+1}}d_ne_n +\varps_n\inn{\xi, e_n}d_{n+1}e_{n+1} +\varps_n^{2} \inn{\xi, e_{n+1}}d_{n+1}e_{n+1}\Big)
\end{align*}
On the other hand, 
\begin{align*}
q_n D\xi &= q_n\Big( \sum_{k=1}^\infty \inn{\xi, e_k} d_ke_k\Big) =  \sum_{k=1}^\infty \inn{\xi, e_k} d_k \inn{e_k, f_n}f_n\\
&= \inn{\xi,e_n}d_n\inn{e_n, f_n}f_n +\inn{\xi, e_{n+1}}d_{n+1} \inn{e_{n+1}, f_n}f_n\\
&= \frac{1}{1+ \varps_n^2}\Big(\inn{\xi,e_n}d_ne_n +\inn{\xi, e_n}d_n\varps_ne_{n+1} +\inn{\xi, e_{n+1}}d_{n+1}\varps_n e_n +\inn{\xi, e_{n+1}}d_{n+1}\varps_n^2 e_{n+1} \Big)
\end{align*}
Subtracting the two expressions above yields
\begin{align}\label{eq:commutator-formula}
[D,q_n](\xi)=\frac{\varps_n}{1+\varps_n^2}\Big( \inn{\xi,e_{n+1}}(d_n-d_{n+1})e_n + \inn{\xi, e_n}(d_{n+1}-d_n)e_{n+1} \Big),
\end{align}
and thus
\begin{align*}
\left\|[D,q_n](\xi)\right\|_2^2 \leq \Big( \frac{\varps_n}{1+\varps_n^2} (d_{n+1}-d_n)\Big)^2\|\xi\|_2^2.
\end{align*}
It therefore follows that $[D,q_n]$ extends boundedly and furthermore that 
\begin{align}\label{eq:norm-estimate}
\left\|[D,q_n] \right\| \leq  \frac{\varps_n}{1+\varps_n^2} (d_{n+1}-d_n) \underset{n\to \infty}{\longrightarrow} 0,
\end{align}
by the choice of the sequence $(\varps_n)_{n}$.\\
Next note that $\inn{f_{2n}, f_{2m}}=0$ whenever $n\neq m$, and setting
\[
\A:= \text{span}_{\CC} \left\{1, q_{2n} \mid n\in \NN\right\} \subset \BB(H)
\]
therefore defines a unital, commutative $*$-subalgebra of $\BB(H)$. Moreover, 
every  $a\in \A$ preserves the domain of $D$ since this was shown to be the case for each $q_{k}$ above, and by \eqref{eq:norm-estimate} it follows that the commutator $[D,a]$ extends boundedly. Hence, $(\A, H, D)$ constitutes a commutative spectral triple and we now  prove that it satisfies (c) and (d).\\

To prove (c), assume that $a=\sum_{i=1}^l \alpha_{2n_i} q_{2n_i}\in \A$ commutes with $D$ and assume, without loss of generality, that $n_1<n_2 <\cdots < n_l$. We then need to prove that $a=0$. From \eqref{eq:commutator-formula} it follows that
\[
[D,q_n](e_k)= 
\begin{cases}
  \frac{\varps_n}{1+\varps_n^2}(d_{n+1}-d_n)e_{n+1}   &  \text{ if } k=n \\
  \frac{\varps_n}{1+\varps_n^2}(d_{n}-d_{n+1})e_{n}  &  \text{ if } k= n+1 \\
  0  &  \text{ otherwise }
\end{cases}
\]
Hence
\[
[D, a](e_{2n_l+1})=\alpha_{2n_l}  \frac{\varps_{2n_l}}{1+\varps_{2n_l}^2}(d_{2n_l}-d_{2n_l+1})e_{2n_l}, 
\]
and since $\varps_{2n_l}>0$ and $d_{2n_l+1}> d_{2n_l}$ the assumption $[D,a]=0$ forces $\alpha_{2n_l}=0$.  Repeating the argument just given, we conclude that $\alpha_{2n_i}=0$ for all $i\in \{1,\dots, l\}$ as desired.\\

We now only need to prove (d), i.e.~that  the extended metric $d_D$ on the state space of the $C^*$-closure,  $A$, of ${\A}$, attains the value infinity.
To see this,  define
\[
\delta_n:=  \frac{\varps_n}{1+\varps_n^2}(d_{n+1}-d_n) \quad \text{ and } \quad a_n:= \frac{1}{\delta_n}q_n.
\]
Then $a_{2n}\in \A$ and by \eqref{eq:norm-estimate} we also have $\|[D,a_{2n}]\|\leq 1$. Consider now the vector state  on $\BB(H)$ defined by
\[
T\longmapsto \sum_{n=1}^\infty 2^{-n}\inn{Tf_{2n}, f_{2n}},
\]
and its restriction, $\varphi$, to $A$. Note that $\varphi$ has the property that
\[
\varphi(a_{2n})=2^{-n}\delta_{2n}^{-1} =\frac{1+\varps_{2n}^2}{\varps_{2n} 2^n (d_{2n+1}-d_{2n})} \underset{n\to \infty }{\longrightarrow} \infty,
\]
where the divergence follows by our choice of the sequence $(\varps_n)_n$ in \eqref{eq:varps-sequence}.  Denote by $\psi$ the restriction to $A$ of the vector state on $\BB(H)$ defined by $T\mapsto \inn{Tf_1, f_1}$ and note that $\psi(a_{2n})=0$ for all $n\in \NN$. Hence
\[
d_D(\varphi,\psi)\geq \sup\{ |\varphi(a_{2n})-\psi(a_{2n})| \mid n\in \NN \}=\sup\{2^{-n}\delta_{2n}^{-1} \mid n\in \NN \}=\infty.
\]
This concludes the proof of Theorem \ref{thm:main-thm-B}.
\end{proof}


\end{document}